\documentclass{article}

\usepackage[utf8]{inputenc}
\usepackage{amsmath,amsthm,amssymb}
\usepackage{thmtools}
\usepackage{mathrsfs}
\usepackage{mathtools}
\usepackage{fullpage}
\usepackage{enumitem}

\newtheorem{thm}{Theorem}[section]
\newtheorem{lem}[thm]{Lemma}
\newtheorem{cor}[thm]{Corollary}

\theoremstyle{definition}			                						
\newtheorem{mydef}[thm]{Definition}
\newtheorem{remark}[thm]{Remark}
\newtheorem{example}[thm]{Example}

\newtheorem{obs}[thm]{Observation}

\newcommand{\floor}[1]{\left\lfloor #1 \right\rfloor}

\title{Polynomials that preserve nonnegative monomial matrices}
\author{Benjamin J.~Clark 
\and Pietro Paparella}
\date{January 2, 2024}

\begin{document}
\maketitle

\begin{abstract}
    A recently-established necessary condition for polynomials that preserve the class of entrywise nonnegative matrices of a fixed order is shown to be necessary and sufficient for the class of nonnegative monomial matrices. Along the way, we provide a formula for computing an arbitrary power of a monomial matrix and a formula for computing the polynomial of a nonnegative monomial matrix.
\end{abstract}

\section{Introduction}

Motivated by the \emph{nonnegative inverse eigenvalue problem}, Loewy and London \cite{ll1978-79} asked for a characterization of 
\[ \mathscr{P}_n \coloneqq \{ p \in \mathbb{R}[t] \mid p(A) \ge 0,\ \forall A \in \mathsf{M}_n(\mathbb{R}),\ A \ge 0 \}. \]
Clearly, $\mathscr{P}_n$ contains polynomials with nonnegative coefficients, but it is known that it contains polynomials having some negative entries.

The characterization of $\mathscr{P}_1$ is known as the P\'{o}lya--Szeg\"{o} theorem, which asserts that $p \in \mathscr{P}_1$ if and only if
\[ p(t) = \left( f_1(t)^2 + f_2(t)^2 \right) + t \left( g_1(t)^2 + g_2(t)^2 \right), \]
where $f_1,f_2,g_1,g_2 \in \mathbb{R}[t]$ (see, e.g., Powers and Reznick \cite{pr2000}). 

Bharali and Holtz \cite{bh2008} gave partial results for the superset
\[ \mathscr{F}_n \coloneqq \{ f\ {\rm entire} \mid f(A) \ge 0,\ \forall A \in \mathsf{M}_n(\mathbb{R}),\ A \ge 0 \} \]
of $\mathscr{P}_n$, including a characterization of two-by-two matrices and results for certain structured nonnegative matrices, including upper-triangular matrices and circulant matrices.

More recently, Clark and Paparella \cite{clark2021polynomials} provided novel necessary conditions for $\mathscr{P}_n$. In particular, they showed that the coefficients of a polynomial belonging to $\mathscr{P}_n$ satisfy certain linear inequalities. It was also shown \cite[Corollary 4.5]{clark2021polynomials} that if $p \in \mathscr{P}_n$, $r \in \{ 0,\ldots,n-1 \}$, $m \coloneqq \deg p$, and $\mathcal{I}_{(m,n,r)} \coloneqq \{ 0 \le k \le m \mid k \bmod{n} = r \}$, then 
\[ p_{r}(t) = p_{(r,n)}(t) \coloneqq \sum_{k \in \mathcal{I}_{(m,n,r)}} a_k t^k \in \mathscr{P}_1. \] 
Since $\mathscr{P}_{k+1} \subseteq \mathscr{P}_{k}$, $\forall k \in \mathbb{N}$ (see, e.g., Bharali and Holtz \cite[Lemma 1]{bh2008}), it follows that 
\begin{equation}
    \label{necessary_condition}
    p_{(r,k)} \in \mathscr{P}_1,\ \forall k \in \{1,\ldots,n\},\ \forall r \in \{0,\ldots,k-1\}.
\end{equation}

In this work, condition \eqref{necessary_condition} is shown to be sufficient for the class of nonnegative monomial matrices. In addition, a formula is given for computing the polynomial of a monomial matrix.

\section{Notation}

If $n \in \mathbb{N}$, then $\mathsf{M}_n = \mathsf{M}_n(\mathbb{C})$ denotes the algebra of $n$-by-$n$ matrices with complex entries and $S_n$ denotes the symmetric group of order $n$. If $\sigma \in S_n$ and $x \in \mathbb{C}^n$, then $\sigma(x)$ denotes the vector whose $i$th-entry is $x_{\sigma(i)}$ and $P_\sigma \in \mathsf{M}_n$ denotes the permutation matrix corresponding to $\sigma$ (i.e., $P_\sigma$ is the the $n$-by-$n$ matrix whose $(i,j)$-entry is \(\delta_{\sigma(i),j}\), where $\delta_{ij}$ denotes the Kronecker delta). As is well-known, $\left( P_\sigma \right)^{-1} = P_{\sigma^{-1}} = \left(P_\sigma \right)^\top$. When the context is clear, $P_\sigma$ is abbreviated to $P$.  

If $x \in \mathbb{C}^n$, then $D_x \in \mathsf{M}_n$ denotes the diagonal matrix such that $d_{kk} = x_k$, $1 \le k \le n$. If $A \in \mathsf{M}_n$, then $A$ is called \textit{monomial}, a \textit{monomial matrix}, or a \textit{generalized permutation matrix} if there is an invertible diagonal matrix $D = D_x$ and a permutation matrix $P$ such that $A = D P$. The set of all $n$-by-$n$ monomial matrices is denoted by $\mathsf{GP}_n = \mathsf{GP}_n(\mathbb{C})$. 
If $A = D_x P$, with $x \in \mathbb{C}^n$, then 
\[\alpha_x \coloneqq \det D_x = \prod_{k=1}^n x_k. \]
The restriction of $\mathscr{P}_n$ to the class of nonnegative monomial matrices is denoted by $\mathscr{P}_n^{\rm mon}$. 

If $n \in \mathbb{N}$, then $\langle n \rangle \coloneqq \{ 1,\ldots, n \}$, $\langle n \rangle_0 \coloneqq \{0\} \cup \langle n \rangle$, and $\pi = \pi_n: \langle n \rangle \longrightarrow \langle n \rangle$ is the permutation defined by $\pi(i) = (i\bmod{n}) + 1$. The permutation matrix corresponding to $\pi$ is denoted by $C = C_n$. i.e., \( C = \left[ \delta_{\pi(i),j} \right] \in \mathsf{M}_n \). For example, if $n=3$, then  
\[\pi_3 = 
\begin{pmatrix} 
1 & 2 & 3 \\
2 & 3 & 1
\end{pmatrix}
~\text{and}~ 
C_3 = 
\begin{bmatrix} 
0 & 1 & 0   \\
0 & 0 & 1   \\
1 & 0 & 0 
\end{bmatrix}. \]

If $x \in \mathbb{C}^n$, then 
\begin{equation}
    \label{fundamental_monomial}
        K_x \coloneqq D_x C \in \mathsf{M}_n. 
\end{equation}
For example, if $x \in \mathbb{C}^3$, then 
\[ K_x = 
\begin{bmatrix} 
0   & x_1   & 0     \\
0   & 0     & x_2   \\
x_3 & 0     & 0 
\end{bmatrix}. \]

\begin{mydef}
[{\cite[Definition 3.1]{cp2022}}]
\label{indexsets}
    Let $p(t) = \sum_{k=0}^m a_k t^k \in \mathbb{C}[t]$, $n \in \mathbb{N}$, and $r \in \langle n-1 \rangle_0$. If   
    \[ \mathcal{I}_{(m,n,r)} \coloneqq \{ k \in \langle m \rangle_0 \mid k \bmod{n} = r \}, \]
    then the polynomial
    \begin{equation}
        \label{rmodnpart}
            p_{(r,n)}(t) \coloneqq \sum_{k \in \mathcal{I}_{(m,n,r)}} a_k t^k, 
    \end{equation} 
    is called the \emph{$r \bmod{n}$-part of $p$}.
\end{mydef}

\begin{remark}
    If $m < n - 1$ and $r > m$, then $\mathcal{I}_{(m,n,r)} = \varnothing$. In such a case, the sum on the right-hand side of \eqref{rmodnpart} is empty. Consequently, 
    \[ 
    p_{(r,n)}(t) =
    \begin{cases}
        a_r t^r, & 0 \le r \le m \\
        0, & m < r \le n-1
    \end{cases}.
    \] 
\end{remark}

\begin{obs} 
[{\cite[Observation 3.2]{cp2022}}]
\label{obs:rpartsum}
    If $p \in \mathbb{C}[t]$, then \( p(t) = \sum_{r = 0}^{n-1} p_{(r,n)}(t) \).
\end{obs}

\section{Preliminary results}

Johnson and Paparella \cite[Lemma 3.3]{jp2016} used the \emph{relative gain array} of a matrix to give a short proof of the elementary fact that if $x \in \mathbb{C}^n$, then $D_{\sigma(x)} = P_\sigma D_x P_\sigma^\top$, i.e., 
\begin{align}
    P_\sigma D_x = D_{\sigma(x)} P_\sigma. \label{jplemma} 
\end{align}

\begin{lem} 
[Frobenius normal form for monomial matrices]
\label{lem:frobenius_normal_form}
    If $A \in \mathsf{GP}_n$, then there is a permutation matrix $Q$ such that 
    \begin{align}
        Q^\top A Q 
    = \bigoplus_{i=1}^k D_{y_i} C_{n_i} 
    = \bigoplus_{i=1}^k K_{y_i}, \label{fnfmm}
    \end{align} 
    where $1 \leq k \leq n$, $y = Q^\top x$, and $y_i \in \mathbb{C}^{n_i}$ is the partition of $y$ into $k$ blocks of size $n_i$ for $i \in \langle{k}\rangle$.
\end{lem}

\begin{proof}
By hypothesis, $A = D_x P$ and by the Frobenius normal form for permutation matrices (see, e.g., Paparella \cite[Corollary 4.3]{p2019}), there is a permutation matrix $Q = Q_\gamma$ such that 
\[ Q^\top P Q = \bigoplus_{i=1}^k C_{n_i},\]
where $1 \leq k \leq n$. If $y \coloneqq \gamma^{-1}(x)$, where $y_i \in \mathbb{C}^{n_i}$ is the partition of $y$ into $k$ blocks of size $n_i$ for  $i \in \langle{k}\rangle$, then
\begin{align*}
    Q^\top A Q 
    &= \left(Q^\top D_x \right) P Q     \\
    &= \left( D_y Q^\top \right) P Q    \\
    &= D_y \left(Q^\top P Q\right) 
    = D_y \bigoplus_{i=1}^k C_{n_i} 
    = \bigoplus_{i=1}^k D_{y_i} C_{n_i}
    = \bigoplus_{i=1}^k K_{y_i}. \qedhere
\end{align*}
\end{proof}

\begin{lem}
\label{lem:primpower}
    Let $x \in \mathbb{C}^n$ and let $j$ be a nonnegative integer. If $q \coloneqq \floor{j/n}$ and $r = j\bmod{n}$, then
    \begin{equation}
        K_x^j = \alpha_x^q \left( \prod_{j=0}^{r-1} D_{\pi^j(x)} \right) C^r. \label{Kxpower}
    \end{equation}
\end{lem}

\begin{proof}
Proceed by induction on $j$. For the base-case, if $j = 0$, then $q = r = 0$, and 
\[ K_x^0 = I = \alpha_x^0 \left( \prod_{j=0}^{-1} D_{\pi^j(x)} \right) K_x^0, \]
given that the product on the right-hand side is empty and equal to $I$.

For the induction-step, assume that the result holds when $j = \ell \ge 0$ and write $\ell = qn + r$, where $0 \leq r \le n - 1$. Notice that  
\begin{align*}
    K_x^{\ell+1} = K_x^\ell K_x 
    &= \alpha_x^q \left( \prod_{j=0}^{r-1} D_{\pi^j(x)} \right) C^r D_x C \\
    &= \alpha_x^q \left(\prod_{j=0}^{r-1} D_{\pi^j(x)} \right) D_{\pi^r (x)} C^r C \tag{by \eqref{jplemma}} \\
    &= \alpha_x^q \prod_{j=0}^{r} D_{\pi^j(x)} C^{r+1},
\end{align*}
and if $0 \leq r \le n - 2$, then $r+1 = (\ell+1) \bmod{n}$ and the result follows. 

Otherwise, if $r = n - 1$, then $\ell + 1 = n(q+1)$, $(\ell+1) \bmod{n} = 0$, and
\begin{align*}
    K_x^{\ell+1} 
    = \alpha_x^q \prod_{j=0}^{n-1} D_{\pi^j(x)} C^n 
    = \alpha_x^q \prod_{j=0}^{n-1} D_{\pi^j(x)} C^0.   
\end{align*}
Given that $\vert\pi\vert = n$ (here, $\vert \cdot \vert$ denotes the order of $\pi$ as a group-element of $S_n$), the matrix 
\[ \prod_{j=0}^{n-1} D_{\pi^j(x)} = D_{\prod_{j=0}^{n-1} \pi^j(x)} \]
is a diagonal matrix such that every diagonal entry equals $\alpha_x$. Thus,  
\[ K_x^{\ell+1} = \alpha_x^{q+1} \prod_{j=0}^0 D_{\pi^j(x)} C^0, \]
as desired.
\end{proof}

\begin{remark}
    If $0 \le r \le n-1$, then $\floor{r/n} = 0$ and applying \eqref{Kxpower} in this case yields  
    \[ K_x^r = \left( \prod_{j=0}^{r-1} D_{\pi^j(x)} \right) C^r. \]
    Thus, if $j$ is a nonnegative integer, $q \coloneqq \floor{j/n}$, and $r = j\bmod{n}$, then $K_x^j= \alpha_x^q K_x^r$.
\end{remark}

\section{Main results}

Hereinafter, it is assumed that
\[ p(t) = \sum_{k=0}^m a_k t^k \in \mathbb{R}[t], \]
where $a_m \not = 0$.

The following result gives a closed-form formula for computing the polynomial of a nonnegative monomial matrix. 

\begin{thm}
\label{thm:circpoly}
    If $x > 0$ and $K_x$ is defined as in \eqref{fundamental_monomial}, then 
    \[p(K_x) = \sum_{r=0}^{n-1} \frac{p_{(r,n)}(\alpha_x^{1/n})}{\alpha_x^{r/n}} K_x^r.\]
\end{thm}

\begin{proof}
If $k \in \mathcal{I}_{(m,n,r)}$, then $k = q_k n + r$ where $0 \leq r < n$. By Lemma \ref{lem:primpower} and Observation \ref{obs:rpartsum} we have
\begin{align*}
    p(K_x) &= \sum_{r = 0}^{n-1} p_{(r,n)}(K_x) 
    = \sum_{r = 0}^{n-1} \sum_{k \in \mathcal{I}_{(m,n,r)}} a_k K_x^k 
    = \sum_{r = 0}^{n-1} \sum_{k \in \mathcal{I}_{(m,n,r)}} a_k \alpha_x^{q_k} K_x^r.
\end{align*}
Finally, notice that
\begin{align*}
    \sum_{k \in \mathcal{I}_{(m,n,r)}} a_k \alpha_x^{q_k}
    &= \sum_{k \in \mathcal{I}_{(m,n,r)}} a_k \frac{\alpha_x^{q_k + r/n}}{\alpha_x^{r/n}} \\
    &= \frac{1}{\alpha_x^{r/n}}\sum_{k \in \mathcal{I}_{(m,n,r)}} a_k \alpha_x^{k/n} \\
    &= \frac{p_{(r,n)} (\alpha_x^{1/n})}{\alpha_x^{r/n}},
\end{align*}
as desired.
\end{proof}

\begin{cor} 
    \label{cor:monpoly}
    If $A = D_x P \ge 0$ with Frobenius normal form given by \eqref{fnfmm}, then 
    \begin{align} 
    \label{eq:monpoly}
    p(A) 
    = Q \left(  
    \bigoplus_{i=1}^k \sum_{r=0}^{n_i - 1} \frac{p_{(r,n_i)} \left(\alpha_{x_i}^{1/n_i} \right) }{\alpha_{x_i}^{r/n_i}} K_{x_i}^r 
    \right) Q^\top. 
    \end{align}
\end{cor}

\begin{proof}
    Follows from Lemma \ref{lem:frobenius_normal_form}, Theorem \ref{thm:circpoly}, and noting that for a matrix $A \in \mathsf{M}_n$ and permutation matrix $P$ we have $p \left( PAP^\top \right) = P p(A) P^\top$.
\end{proof}

\begin{example}
If 
    \[
    A = \begin{bmatrix}
        0 & 3 & 0 & 0 \\
        0 & 0 & 5 & 0 \\
        0 & 0 & 0 & 2 \\
        1 & 0 & 0 & 0
    \end{bmatrix}
    \]
and 
    \[
    p(t) = t^{20} + 4t^{15} + 2t^8 + 3t^2 + t + 5,
    \]
then, by Theorem \ref{thm:circpoly},  
    \begin{align*}
        p(A) 
        &= \sum_{r=0}^{3} \frac{p_{(r,4)}(30^{1/4})}{30^{r/4}} A^r                                                                                                  \\
        &= p_{(0,4)}(30^{1/4}) I_4 + \frac{p_{(1,4)}(30^{1/4})}{30^{1/4}} A + \frac{p_{(2,4)}(30^{1/4})}{30^{1/2}} A^2 + \frac{p_{(3,4)}(30^{1/4})}{30^{3/4}} A^3   \\
        &= (30^{5} + 2\cdot 30^2 + 5) I_4 + A + 3 A^2 + 30^3 A^3.
    \end{align*}
\end{example}

\begin{thm} 
    \label{thm:charmon}
        $p \in \mathscr{P}_n^{\rm mon}$ if and only if $p_{(r,k)} \in \mathscr{P}_1,\ \forall k \in \langle n \rangle,\ \forall r \in \langle n - 1 \rangle_0$.
\end{thm}

\begin{proof}
Since condition \eqref{necessary_condition} is necessary for polynomials belonging to $\mathscr{P}_n$ (Clark and Paparella \cite[Corollary 4.9]{clark2021polynomials}), it must also be necessary for polynomials belonging to $\mathscr{P}_n^{\rm mon}$.  

For sufficiency, if $A = D_x P \ge 0$ with Frobenius normal form given by \eqref{fnfmm}, then $p(A) \ge 0$ in view of Corollary \ref{cor:monpoly} and \eqref{eq:monpoly}. 
\end{proof}

\bibliographystyle{abbrv}
\bibliography{monomialmatrices}

\end{document}